\documentclass[preprint]{elsarticle}
\usepackage{geometry}                
\usepackage{graphicx}
\usepackage{amssymb,mathrsfs,amsthm,amscd,amsmath}
\usepackage{epstopdf}
\usepackage{bm} 
\usepackage[usenames,dvipsnames]{xcolor}
\usepackage{hyperref}

\newtheoremstyle{cited}%
  {3pt}
  {3pt}
  {\itshape}
  {}
  {\bfseries}
  {.}
  {.5em}
  {\thmname{#1} \thmnumber{#2} \thmnote{\normalfont#3}}

\newtheoremstyle{cited2}%
  {3pt}
  {3pt}
  {\upshape}
  {}
  {\bfseries}
  {.}
  {.5em}
  {\thmname{#1} \thmnumber{#2} \thmnote{\normalfont#3}}

\newtheorem{thm}{Theorem}[section]

\newtheorem{lemma}[thm]{Lemma}
\newtheorem{prop}[thm]{Proposition}
\newtheorem{corollary}[thm]{Corollary}

\theoremstyle{definition}

\newtheorem{axiom}[thm]{Axiom}
\newtheorem{eg}[thm]{Example}
\newtheorem{defn}[thm]{Definition}

\theoremstyle{remark}
\newtheorem{remark}[thm]{Remark}
\newtheorem{notation}[thm]{Notation}

\theoremstyle{cited}

\newtheorem{citedlemma}[thm]{Lemma}
\newtheorem{citedprop}[thm]{Proposition}

\theoremstyle{cited2}
\newtheorem{citeddefn}[thm]{Definition}

\theoremstyle{empty}
\newtheorem{duplicate}{Theorem}

\newcommand{\sop}{system of parameters}

\newcommand{\CM}{Cohen-Macaulay}
\newcommand{\fg}{finitely-generated}
\newcommand{\ph}{phantom}

\newcommand{\phex}{\ph\  extension}

\newcommand{\lcphex}{lim cl-\ph\ extension}

\newcommand{\dvr}{discrete valuation ring}
\newcommand{\cld}{complete local domain}
\newcommand{\charp}{characteristic $p>0$}
\newcommand{\nzd}{non-zerodivisor}
\newcommand{\axioma}{Functoriality Axiom}
\newcommand{\axiomb}{Semi-residuality Axiom}
\newcommand{\axiomc}{Faithfulness Axiom}

\newcommand{\scr}{\mathscr}

\newcommand{\symm}{\text{Sym}}
\newcommand{\symt}{\text{Sym}^2}
\newcommand{\im}{\text{im}}
\newcommand{\coker}{\text{coker}}
\newcommand{\Hom}{\text{Hom}}
\newcommand{\Ext}{\text{Ext}}
\newcommand{\epf}{\text{epf}}
\newcommand{\id}{\text{id}}
\newcommand{\cl}{\text{cl}}

\newcommand{\tr}{\text{tr}}

\newcommand{\syz}{\text{syz}}

\newcommand{\Z}{\mathbb{Z}}

\title{Closure Operations that Induce Big Cohen-Macaulay Algebras}
\author[rg]{Rebecca~R.G.\corref{cor1}}
\ead{rirg@umich.edu}

\cortext[cor1]{Corresponding author \\ E-mail address: rirg@umich.edu}

\address[rg]{Department of Mathematics, University of Michigan, 2074 East Hall, 530 Church Street, Ann Arbor, MI 48109}

\begin{document}

\begin{abstract}
We study closure operations over a local domain $R$ that satisfy a set of axioms introduced by Geoffrey Dietz. The existence of a closure operation satisfying the axioms (called a Dietz closure) is equivalent to the existence of a big \CM\ module for $R$. When $R$ is complete and has \charp, tight closure and plus closure satisfy the axioms.

We give an additional axiom (the Algebra Axiom), such that the existence of a Dietz closure satisfying this axiom is equivalent to the existence of a big \CM\ algebra. We prove that many closure operations satisfy the Algebra Axiom, whether or not they are Dietz closures. We discuss the smallest big \CM\ algebra closure on a given ring, and show that every Dietz closure satisfying the Algebra Axiom is contained in a big \CM\ algebra closure. This leads to proofs that in rings of \charp, every Dietz closure satisfying the Algebra Axiom is contained in tight closure, and there exist Dietz closures that do not satisfy the Algebra Axiom.
\end{abstract}


\maketitle

\section{Introduction}

The study of big Cohen-Macaulay algebras was originally motivated by the Direct Summand Conjecture \cite{directsummand}. The existence of big \CM\ algebras implies the Direct Summand Conjecture, and with it a number of equivalent conjectures central to commutative algebra, including the Monomidal Conjecture \cite{directsummand} and the Canonical Element Conjecture \cite{canonicalelement}.

The equal characteristic case of these results, as well as the existence of big \CM\ algebras, was proved using tight closure methods \cite{modulemods, faithfullyflat, bigcmalgapps, betterbigcmalgebras}. In mixed characteristic, the homological conjectures were proved in dimension at most 3 by Heitmann \cite{heitmann}, and the existence of big \CM\ algebras by Hochster \cite{dim3bigcmalgebras}, but via a method that could not be extended to higher dimensional rings. In 2016, Yves Andre proved the Direct Summand Conjecture for mixed characteristic rings of arbitrary dimension using perfectoid methods \cite{andredirectsummand}, and used this result to prove the existence of big \CM\ algebras as well.

In \cite{dietz}, Dietz gave a list of axioms for a closure operation such that for a  local domain $R$, the existence of a closure operation satisfying these properties (called a Dietz closure) is equivalent to the existence of a big \CM\ module. The closure operation can be used to show that when module modifications are applied to $R$, the image of 1 in the resulting module is not contained in the image of the maximal ideal of $R$. When $R$ is complete and has \charp, tight closure is a Dietz closure, as are plus closure and solid closure \cite{dietz}.

Dietz asked whether it was possible to give an additional axiom such that the existence of a Dietz closure satisfying this axiom is equivalent to the existence of a big \CM\ algebra. Due to results on the existence of weakly functorial big \CM\ algebras \cite{bigcmalgapps, heitmannma}, big \CM\ algebras are even more desirable than big \CM\ modules. Further, big \CM\ algebras are known to exist in every case where big \CM\ modules are known to exist. 

In this paper, we answer Dietz's question in the positive, by giving an Algebra Axiom, Axiom \ref{algebraaxiom}.  We prove:

\begin{duplicate}[Theorem~\ref{bigcmalgebras}, Corollary~\ref{bigcmalgebraaxiom}]
A local domain $R$ has a Dietz closure that satisfies the Algebra Axiom if and only if $R$ has a big \CM\ algebra.
\end{duplicate}

In Section \ref{algebraaxiomsection}, we prove that many closure operations satisfy this axiom, including tight closure, which is also a Dietz closure, and torsion-free algebra closures, which are not in general Dietz closures. To prove this, we find alternative characterizations of cl-\phex s, maps that resemble splittings with respect to a closure operation cl, for the closures cl that we discuss. We also show that the big \CM\ algebras that we construct using the Algebra Axiom give the same closure operation as those constructed using algebra modifications as in \cite{bigcmalgapps}, and that this closure operation is the smallest big \CM\ algebra closure on any ring $R$.

In Section \ref{dietzplusalgincharp}, we use these results to compare Dietz closures satisfying the Algebra Axiom to other closure operations. We prove that each such closure is contained in a big \CM\ algebra closure. In \charp, we also show that all such closures are contained in tight closure. These results may lead to further characterizations of the singularities of a ring in terms of these closure operations. We also give an example that shows that not all Dietz closures satisfy the Algebra Axiom.

Since Dietz closures satisfying the Algebra Axiom, in particular big \CM\ algebra closures, share many of the key properties of tight closure, they may be a useful alternative to tight closure. This is particularly true for the mixed characteristic case, where there is no version of tight closure yet. In particular, in Definition \ref{testidealdef}, we define the test ideal of an arbitrary closure operation. This extends the definition of the tight closure test ideal that has been the subject of much study (see \cite{testidealssurvey} for a survey).

\section{Background}
\label{background}

All rings will be commutative Noetherian rings with unity, unless otherwise specified.

\begin{defn}
A \textit{closure operation} cl on a ring $R$ is a map $N \to N_M^\cl$ of submodules $N$ of \fg\ $R$-modules $M$ such that if $N \subseteq N' \subseteq M$ are \fg\ $R$-modules,
\begin{enumerate}
\item (Extension) $N \subseteq N_M^\cl$,
\item (Idempotence) $(N_M^\cl)_M^\cl=N_M^\cl$, and
\item (Order-Preservation) $N_M^\cl \subseteq (N')_M^\cl$.
\end{enumerate}
\end{defn}

\begin{defn}
\label{moduleclosure}
Suppose that $S$ is an $R$-module (resp. $R$-algebra). We can define a closure operation $\cl_S$ on $R$ by
\[u \in N_M^{\cl_S} \text{ if for all } s \in S, s \otimes u \in \im(S \otimes N \to S \otimes M),\]
where $N \subseteq M$ are \fg\ $R$-modules and $u \in M$. This is called a \textit{module (resp. algebra) closure}.
\end{defn}

\begin{remark}
Note that if $S$ is an $R$-algebra, $u \in N_M^{\cl_S}$ if and only if \[1 \otimes u \in \im(S \otimes N \to S \otimes M).\]
\end{remark}

\begin{citeddefn}[{\cite[Lemma~3.4]{dietzclosureprops}}]
\label{algebrafamilyclosure}
If $\cal{S}$ is a directed family of $R$-algebras, then we can define a closure operation $\cl_{\cal{S}}$ by $u \in N_M^{\cl_{\cal{S}}}$ if for some $S \in \cal{S}$, $u \in N_M^{\cl_S}$.
\end{citeddefn}

When $\cal{S}$ is any family of $R$-modules, it still generates a closure operation:

\begin{defn}
\label{familygeneratedclosure}
Let $\cal{S}$ be a family of $R$-modules. For $N \subseteq M$ \fg\ $R$-modules, we define $\cl_{\cal{S}}$ as follows:
\begin{enumerate}
\item Let $N_M^{\cl_1}$ be the submodule of $M$ generated by the elements $u \in M$ such that $u \in N_M^{\cl_S}$ for some $S \in \cal{S}$.
\item Let $N_M^{\cl_{k+1}}=(N_M^{\cl_k})_M^{\cl_1}$.
\item Since $R$ is Noetherian, this will eventually stabilize. Set $N_M^{\cl_{\cal{S}}}$ equal to the stable value of this chain.
\end{enumerate}
\end{defn}

\begin{citeddefn}[{\cite[Definition~2.2]{dietz}}]
\label{phantomDef}
Let $R$ be a ring with a closure operation cl, $M$ a \fg\ $R$-module, and $\alpha:R \to M$ an injective map with cokernel $Q$. We have a short exact sequence 
\[\begin{CD}
0 @>>> R @>{\alpha}>> M @>>> Q @>>> {0.}
\end{CD}\]
Let $\epsilon \in \Ext_R^1(Q,R)$ be the element corresponding to this short exact sequence via the Yoneda correspondence. We say that $\alpha$ is a cl-\phex\ if $\epsilon \in 0^\cl_{\Ext_R^1(Q,R)}$. Equivalently, if $P_\bullet$ is a projective resolution of $Q$ and $\vee$ denotes $\Hom_R(-,R)$, a cocycle representing $\epsilon$ is in $P_1^\vee$ is in $\im(P_0^\vee \to P_1^\vee)_{P_1^\vee}^\cl$.
\end{citeddefn}

\begin{remark}
This definition is independent of the choice of $P_\bullet$ \cite[Discussion~2.3]{dietz}, as is clear from the version of the definition involving Ext.

A split map $\alpha:R \to M$ is cl-\ph\ for any closure operation cl: in this case, the cocycle representing $\epsilon$ is in $\im(P_0^\vee \to P_1^\vee)$. We can view cl-\phex s as maps that are ``almost split" with respect to a particular closure operation.
\end{remark}

\begin{notation}
\label{notation1}
We use some notation from \cite{dietz}. Let $R$ be a ring, $M$ a finitely generated $R$-module, and $\alpha:R \to M$ an injective map with cokernel $Q$. Let $e_1=\alpha(1),e_2, \ldots,e_n$ be generators of $M$ such that the images of $e_2,\ldots,e_n$ in $Q$ form a generating set for $Q$. We have a free presentation for $Q$,
\[
\begin{CD}
{R^m} @>{\nu}>> {R^{n-1}} @>{\mu}>> Q @>>> {0,}
\end{CD}
\]
where $\mu$ sends the generators of $R^{n-1}$ to $e_2,\ldots,e_n$ and $\nu$ has matrix $(b_{ij})_{2 \le i \le n, 1 \le j \le m}$ with respect to some basis for $R^m$. We have a corresponding presentation for $M$, 
\[
\begin{CD}
{R^m} @>{\nu_1}>> {R^n} @>{\mu_1}>> {M,}
\end{CD}
\]
where $\mu_1$ sends the generators of $R^n$ to $e_1,\ldots,e_n$. Using the same basis for $R^m$ as above, $\nu_1$ has matrix $(b_{ij})_{1 \le i \le n, 1 \le j \le m}$ where $b_{1j}e_1+b_{2j}e_2+\ldots+b_{nj}e_n=0$ in $M$ \cite[Discussion~2.4]{dietz}. The top row of $\nu_1$ gives a matrix representation of the map $\phi:R^m \to R$ in the following diagram:

\[
\begin{CD}
0 @>>> R @>{\alpha}>> M @>>> Q @>>> 0 \\
@. @A{\phi}AA @A{\psi}AA @A{\id_Q}AA @AAA \\
 @. {R^m} @>{\nu}>> {R^{n-1}} @>{\mu}>> Q @>>> 0 \\
\end{CD}
\]
\end{notation}

In \cite[Discussion 2.4]{dietz}, Dietz gives an equivalent definition of a \phex\ using the free presentations $M$ and $Q$ given above. While he assumes that $R$ is a \cld\ and that cl satisfies 2 additional properties, these are not needed for all of the results. We restate some of his results in greater generality below.

\begin{citedlemma}[{\cite[Lemma~2.10]{dietz}}]
\label{lemma2.10}
Let $R$ be a ring possessing a closure operation $\cl$. Let $M$ be a finitely generated module, and let $\alpha:R \to M$ be an injective map. Let notation be as above. Then $\alpha$ is a cl-\phex\ of $R$ if and only if the vector $(b_{11},\ldots,b_{1m})^\tr$ is in $B_{R^m}^\cl$, where $B$ is the $R$-span in $R^m$ of the vectors $(b_{i1},\ldots,b_{im})^\tr$ for $2 \le i \le n$.
\end{citedlemma}

Below we give the definition of a Dietz closure.

\begin{defn}
\label{axioms}
Let $(R,m)$ be a fixed local domain and let $N, M,$ and $W$ be arbitrary finitely generated $R$-modules with $N \subseteq M$. A closure operation cl is called a \textit{Dietz closure} if the following conditions hold:
\begin{enumerate}
\item (Functoriality) Let $f:M \to W$ be a homomorphism. Then $f(N_M^\cl) \subseteq f(N)_W^\cl$.
\item (Semi-residuality) If $N_M^\cl=N$, then $0_{M/N}^\cl=0$.
\item (Faithfulness) The maximal ideal $m$ is closed in $R$.
\item (Generalized Colon-Capturing) Let $x_1,\ldots,x_{k+1}$ be a partial system of parameters for $R$, and let $J=(x_1,\ldots,x_k)$. Suppose that there exists a surjective homomorphism $f:M \to R/J$ and $v \in M$ such that $f(v)=x_{k+1}+J$. Then $(Rv)_M^\cl \cap \ker f \subseteq (Jv)_M^\cl$.
\end{enumerate}
\end{defn}

\begin{remark}
The axioms originally included the assumption that $0^\cl_R=0$, but this is implied by the other axioms \cite{newdietz}.

A closure operation on any ring $R$ can satisfy the \axioma, the \axiomb, or both. A closure operation on any local ring $R$ can satisfy the \axiomc.
\end{remark}

The proof of the next lemma requires $Q$ to have a minimal generating set, so we assume that $R$ is local for this generalization of \cite[Lemma~2.11]{dietz}:

\begin{lemma}
\label{lemma2.11}
Let $(R,m)$ be a local ring possessing a closure operation $\cl$ satisfying the \axioma, the \axiomb, and the \axiomc, and such that $0^\cl_R=0$. If $M$ is a finitely generated $R$-module such that $\alpha:R \to M$ is cl-phantom, then $\alpha(1) \not\in mM$.
\end{lemma}

We use a result on \ph\  extensions from \cite[Section~5]{hhpaper} that uses the notation of Notation \ref{notation1}.

\begin{lemma}\cite[Lemma~5.6a~and~c]{hhpaper}
\label{lemma5.6}
Let 
\[\begin{CD}
0 @>>> R @>{\alpha}>> M @>>> Q @>>> 0 \\
\end{CD}\]
be an exact sequence. Letting $P_\bullet$ be a projective resolution for $Q$, we get a commutative diagram 
\[\begin{CD}
0 @>>> R @>{\alpha}>> M @>>> Q @>>> 0 \\
@AAA @AA{\phi}A @AAA @AA{\id}A @. \\
P_2 @>>> P_1 @>{d}>> P_0 @>>> Q @>>> 0. \\
\end{CD}\]
By definition, $\alpha$ is cl-\ph\  if and only if $\phi \in \im(\Hom_R(P_0,R) \to \Hom_R(P_1,R))^\cl_{\Hom_R(P_1,R)}.$
\begin{enumerate}
\item For each $c \in R$, the image of $c\phi$ is a coboundary in $H^1(\Hom_R(P_\bullet,R))$ if and only if there is a map $\gamma:M \to R$ such that $\gamma \alpha=c(\id_R)$.
\item Let $S$ be an $R$-algebra, and $G_\bullet$ a projective resolution for $S \otimes_R Q$ that ends 
\[\ldots \to S \otimes P_1 \to S \otimes P_0 \to S \otimes Q \to 0.\]
The sequence
\[\begin{CD}
0 @>>> R @>{\alpha}>> M @>>> Q @>>> 0 \\
\end{CD}\]
remains exact upon tensoring with $S$ if and only if $\id_S \otimes_R \phi \in \Hom_S(S \otimes_R P_1,S)$ is a 1-cocycle in $\Hom_S(G_\bullet,S)$, in which case $\id_S \otimes_R \phi$ represents the extension over $S$ given by the sequence
\[\begin{CD}
0 @>>> S @>{\id_S \otimes \alpha}>> S \otimes_R M @>>> S \otimes_R Q @>>> 0. \\
\end{CD}\]
\end{enumerate}
\end{lemma}

\section{An axiom for big \CM\ algebras}
\label{algebraaxiomsection}

In \cite{dietz}, Dietz asked whether it was possible to give a characterization of Dietz closures that induced big \CM\ algebras. Below, I answer this question positively.

There are many reasons to prefer big \CM\ algebras to big \CM\ modules; one is the ability to compare big \CM\ algebra closures on a family of rings. Suppose that we have the following commutative diagram:
\[
\begin{CD}
B @>>> C \\
@AAA @AAA \\
R @>>> S \\
\end{CD}
\]
with $R \to S$ a local map of local domains, $B$ an $R$-algebra, and $C$ an $S$-algebra. Then if $u \in N_M^{\cl_B}$, $1 \otimes u \in (S \otimes_R N)_{S \otimes_R M}^{\cl_C}.$ This property is a special case of \textit{persistence for change of rings}.

\begin{proof}[Proof of Persistence]
By assumption, $1 \otimes u \in \im(B \otimes_R N \to B \otimes_R M)$. We show that
\[1 \otimes (1 \otimes u) \in \im(C \otimes_S (S \otimes_R N) \to C \otimes_S (S \otimes_R M)).\]
We can identify $C \otimes_S (S \otimes_R N)$ with $C \otimes_R N$, and $C \otimes_S (S \otimes_R M)$ with $C \otimes_R M$. Under these identifications, $1 \otimes (1 \otimes u) \mapsto 1 \otimes u$. So our goal is to show that
\[1 \otimes u \in \im(C \otimes_R N \to C \otimes_R M).\]
There is some element $d \in B \otimes_R N$ that maps to $1 \otimes u$ in $B \otimes_R M$. Then $d \mapsto d'$, an element of $C \otimes_R N.$ Then by the commutativity of the diagram
\[
\begin{CD}
B \otimes_R N @>>> B \otimes_R M \\
@VVV @VVV \\
C \otimes_R N @>>> C \otimes_R M, \\
\end{CD}
\]
$d'$ is an element of $C \otimes_R N$ that maps to $1 \otimes u \in C \otimes_R M$. Thus $1 \otimes u \in (S \otimes_R N)_{S \otimes_R M}^{\cl_C}$.
\end{proof}

This implies that big \CM\ algebra closures are persistent in any case where we can build a commutative diagram as above, with $B$ and $C$ big \CM\ algebras. By a result of Hochster and Huneke \cite[Discussion and Definition 3.8, Theorem 3.9]{bigcmalgapps}, we have such a diagram when $R$ and $S$ are of equal characteristic and $R \to S$ is \textit{permissible}. A local homomorphism $R \to S$ is permissible if every minimal prime of $\hat{S}$ such that $\dim(\hat{S}/Q)=\dim(\hat{S})$ lies over a prime $P$ of $\hat{R}$ such that $P$ contains a minimal prime $P'$ of $\hat{R}$ such that $\dim(\hat{R}/P')=\dim(\hat{R})$.

Suppose that we add the following to the list of axioms for a Dietz closure cl:

\begin{axiom}[Algebra Axiom]
\label{algebraaxiom}
If $R \overset{\alpha}\to M,$ $1 \mapsto e_1$ is cl-phantom, then the map $R \overset{\alpha'}\to \symt(M)$, $1 \mapsto e_1 \otimes e_1$ is cl-phantom. 
\end{axiom}

\begin{remark}
\label{actualsymmetricpowers}
While the Algebra Axiom as stated uses a map to $\symt(M)$, the axiom is more easily understood using the isomorphism to $\symm^{\le 2}(M)/(1-e_1)\symm^{\le 1}(M)$.
This is the module consisting of all elements of $\symm(M)$ in $R$, $M$, and $\symt(M)$, with the following relations: for $r \in R$, $r \sim re_1 \sim re_1 \otimes e_1$, and for $m \in M$, $m \sim m \otimes e_1$. The map $\alpha'$ in the Algebra Axiom is cl-phantom if and only if the map $R \to \symm^{\le 2}(M)/(1-e_1)\symm^{\le 1}(M)$ sending $1 \mapsto e_1 \otimes e_1$ is cl-phantom, since these modules are isomorphic. To see that the modules are isomorphic, first notice that we can identify any element of $\symm^{\le 2}(M)/(1-e_1)\symm^{\le 1}(M)$ with an element of $\symt(M)$ by tensoring with copies of $e_1$. Next we show that
\[(1-e_1)\symm^{\le 1}(M) \bigcap \symt(M)=0.\] Given $m \in M$, $(1-e_1)m= m - e_1 \otimes m \in \symt(M)$. Since $m \in \symm^1(M)$ and $m \otimes e_1 \in \symt(M)$, the only way for them to be equal is to have $m=0$. This works similarly for $r \in R$. So we have the desired isomorphism.
This holds when we replace 2 by $2^k$ for any $k \ge 0$ as well.

The axiom will be used to show that when we take the direct limit of the $\symm^{2k}(M)$, the image of 1 stays out of the image of $m$. When we view this direct limit as a direct limit of the $\symm^{\le 2^k}(M)/(1-e_1)\symm^{\le 2^k-1}(M)$, we get 
\[\begin{aligned}
\varinjlim \symm^{\le 2^k}(M)/(1-e_1)\symm^{\le 2^k-1}(M)
&=\varinjlim \symm^{\le n}(M)/(1-e_1)\symm^{\le n-1}(M) \\
&=\symm(M)/(1-e_1)\symm(M).
\end{aligned}\]
\end{remark}

\begin{thm}
\label{bigcmalgebras}
If a local domain $R$ has a Dietz closure cl that satisfies the Algebra Axiom, then $R$ has a big \CM\ algebra.
\end{thm}

\begin{remark}
\label{sympowersremark}
Note that if $S$ is an $R$-algebra, and we have an $R$-module $M$ and an $R$-module map $f:M \to S$, we can extend the map to a map from $\symt(M) \to S$ via $m \otimes n \mapsto f(m)f(n)$. If we also have a map $R \to M$, $1 \mapsto e$, then we can extend $f$ to a map from
\[\symm^{\le 2^k}(M)/(1-e)\symm^{\le 2^k-1}(M)\]
to $S$, since $f(e)$ is equal to the image of 1 in $S$ under the composition of maps $R \to M \to S$.
\end{remark}

\begin{proof}
We construct a big \CM\ module $B_1$ as in \cite{dietz} with a map $R \to B_1$, $1 \mapsto e$, and then take $\symm(B_1)/(1-e)\symm(B_1)$. We repeat these two steps infinitely many times, and take the direct limit $B$. This will be an $R$-algebra such that every system of parameters on $R$ is a (possibly improper) regular sequence on $B$. We need to show that $mB \ne B$.

At any intermediate stage $M$, after we have applied module modifications and taken symmetric powers, there is always a map $R \to M$ that factors through all previous intermediate modules. It suffices to show that $\im(1) \not\in mM$. By the arguments of \cite[Theorem~5.1]{dietzclosureprops}, \cite[Proposition 3.7]{bigcmalgapps}, and Remark \ref{sympowersremark}, if $\im(1) \in mB$ then there is some $M$ obtained from $R$ by a finite sequence of module modifications and finite symmetric powers as in Remark \ref{actualsymmetricpowers} for which $\im(1) \in mM$. However, $M$ is a cl-phantom extension of $R$ by \cite{dietz} and the Algebra Axiom. Thus Lemma \ref{lemma2.11} implies that $\im(1) \not\in mM$. Hence $\im(1) \not\in mB$, which implies that $B$ is a big \CM\ algebra for $R$.
\end{proof}

\subsection{A description of the Algebra Axiom in terms of a presentation of $\symt(M)$}

Let $\alpha:R \to M$ be an injective map sending $1 \mapsto e_1$. We use the notation of Notation \ref{notation1}. 
In particular, $Q=\coker(\alpha)$ and $B$ is the matrix $(b_{ij})_{1 \le i \le n, 1 \le j \le m}$ of the map $\nu_1$ with respect to the basis $e_1,\ldots,e_n$ of $R^n$ and the chosen basis of $R^m$. We have a map \[\alpha':R \to \symt(M),\] taking $1 \mapsto e_1 \otimes e_1$. Denote the cokernel by $Q'$. This is isomorphic to $\symt(M)/(R(e_1 \otimes e_1))$.

To get a presentation for $\symt(M)$, we start with the map $R^{m^2} \to R^{n^2}$ given by the matrix $B \otimes B$, and then add in the columns needed for the symmetry relations. There are $\frac{n^2-n}{2}$ of these columns, one for each pair $i < j$, with an entry equal to 1 in the row corresponding to $e_i \otimes e_j$, an entry equal to -1 in the row corresponding to $e_j \otimes e_i$, and 0's elsewhere. Call the corresponding map $\nu'_1$.

To get a presentation for $Q'$, we use this matrix with the top row removed. Call this matrix $\nu'$. We use this presentation to get the following diagram:

\[
\begin{CD}
0 @>>> R @>{\alpha'}>> {\symt(M)} @>>> {Q'} @>>> 0 \\
@AAA @A{\phi'}AA @A{\psi'}AA @A{\id}AA \\
{F_2} @>>> {R^{m^2+\frac{n^2-n}{2}}} @>{\nu'}>> {R^{n^2-1}} @>{\mu'}>> {Q'} @>>> 0
\end{CD}
\]

Let $\oplus$ denote horizontal concatenation of matrices, and $B_i$ the $i$th row of the matrix $B$. The map $\phi'$ is given by the row matrix $(B_1 \otimes B_1) \oplus 0^{\frac{n^2-n}{2}}$, i.e., \[(b_{11}B_1\ b_{12}B_1\ \ldots \ b_{1n}B_1\  0\  \ldots \ 0),\] which is the first row of $\nu_1'$.

The map $\alpha'$ is \ph\  if and only if $\im({\phi'}^{\tr}) \subseteq \im({\nu'}^{\tr})^\cl$. We can rewrite this statement as:

\[\begin{aligned}
&(B_1 \otimes B_1) \oplus 0^{\frac{n^2-n}{2}} \in \left(\sum_{i=2}^n \left((B_i \otimes B_i) \oplus 0^{\frac{n^2-n}{2}}\right)\right. \\
&\left. + \sum_{1 \le i < j \le n} \left((B_i \otimes B_j) \oplus f_{i,j}- (B_j \otimes B_i) \oplus f_{i,j}\right) \right)^\cl_{R^{m^2+\frac{n^2-n}{2}}},
\end{aligned}\]

\noindent where $f_{i,j}$ is the vector of length $\frac{n^2-n}{2}$ with an entry equal to 1 in the $\left(\sum_{\ell=1}^{i-1} (n-\ell)\right)+(j-i)$th 
spot and 0's elsewhere.

\subsection{Proofs that the Algebra Axiom holds for many closure operations}

\subsubsection{Tight Closure}

Let $R$ be a reduced ring of \charp, and let $*$ denote tight closure.

For tight closure, we prove the axiom using the following equivalent definition of a \phex:

\begin{citedprop}[{\cite[Proposition~5.8]{hhpaper}}]
Given a short exact sequence $0 \to R \overset{\alpha}\to M \to Q \to 0$, $\alpha$ is *-\ph\  if and only if there is some $c \in R^\circ$ such that for all sufficiently large $e$, there exist maps $\gamma_e:F^e(M) \to F^e(R)=R$ such that $\gamma_e \circ F^e(\alpha)=c \cdot \id_R$.
\end{citedprop}

\begin{prop}
If an injective map $\alpha:R \to M$ sending $1 \mapsto u$ is *-phantom, then so is the map $\alpha':R \to \symt(M)$ sending $1 \mapsto u \otimes u$. As a result, * satisfies the Algebra Axiom.
\end{prop}

\begin{proof}
Since $R \overset{\alpha}\to M$ is *-phantom, we have maps $\gamma_e$ as described above. Notice that $F^e(\symt(M))$ is the symmetric tensor product $\symt(F^e(M))$. For any $e$ for which $\gamma_e$ exists, define a map 
\[\delta_e:F^e(\symt(M)) \to R\]
 by $\delta_e(m^q \otimes n^q)=\gamma_e(m^q)\gamma_e(n^q)$. (To see that this is well-defined, define it from the tensor product first, then notice that $\delta_e(m^q \otimes n^q)=\delta_e(n^q \otimes m^q)$.) Since $\delta_e(F^e(\alpha))(1)=\delta_e(e_1^q \otimes e_1^q)=c^2$, and $c$ does not depend on the choice of $e$, $R \overset{\alpha'}\to \symt(M)$ is *-phantom.
\end{proof}

\subsubsection{Algebra Closures}

Let $R$ be a ring and $\scr{A}$ a directed family of $R$-algebras, so that given $A,A' \in \scr{A}$, there is a $B \in \scr{A}$ that both $A$ and $A'$ map to. We can define a closure operation using $\scr{A}$ as in Definition \ref{algebrafamilyclosure}. Note that we do not need the elements of $\scr{A}$ to be \fg, and we do not assume that they are \fg\ in this section.

\begin{eg}
All algebra closures $\cl_A$ are closures of this type, with $\scr{A}=\{A\}$. In particular, if $R$ is a domain, plus closure is a closure of this type, with $\scr{A}=\{R^+\}$.
\end{eg}

\begin{eg}
If $R$ is a complete local domain and we let $\scr{A}$ be the set of solid algebras of $R$ (algebras $A$ such that $\Hom_R(A,R) \ne 0$), we get solid closure \cite{solidclosure}.
\end{eg}

To show that the axiom holds for algebra closures, we give an equivalent definition of cl-\ph\  for these closures that is easier to work with.

\begin{lemma}
\label{staysinjective}
Let $\alpha:R \hookrightarrow M$ be an $R$-module homomorphism. Let $A$ be an $R$-algebra, and $W$ the multiplicative set of \nzd s of $R$. If
\begin{enumerate}
\item all elements of $W$ are \nzd s on $A$, and
\item $W^{-1}A$ embeds in a free $W^{-1}R$-module,
\end{enumerate}
then $\id_A \otimes \alpha:A \to A \otimes_R M$ is injective. In particular, if $R$ is a domain and $A$ is a torsion-free algebra over $R$, then $A \hookrightarrow A \otimes_R M$.
\end{lemma}

\begin{proof}
For all \fg\ $R$-submodules $A'$ of $A$, $W^{-1}A'$ embeds in a free $W^{-1}R$-module, $F$. The map $W^{-1}R \to W^{-1}M$ is injective, and this holds when we replace $W^{-1}R$ by $F$. Since $F \to F \otimes_R M$ is injective and $W^{-1}A'$ can be viewed as a submodule of $F$, $W^{-1}A' \to W^{-1}A' \otimes_R M$ is also injective--any element in the kernel would also be in the kernel of $F \to F \otimes M$. Since elements of $W$ are \nzd s on $A$, and hence on $A'$, this implies that $A' \to A' \otimes_R M$ is injective. Hence $A \to A \otimes_R M$ is injective.
\end{proof}

\begin{prop}
Suppose that every $A \in \scr{A}$ satisfies the hypotheses of Lemma \ref{staysinjective}. Then an injective map $\alpha:R \to M$ is cl-\ph\  if and only if for some $A \in \scr{A}$ there is a map $\gamma:A \otimes M \to A$ such that $\gamma \circ (\id_A \otimes \alpha)=\id_A$, i.e., if and only if $\id_A \otimes \alpha$ splits.
\end{prop}

\begin{proof}
By Lemma \ref{staysinjective}, $\id_A \otimes \alpha:A \to A \otimes_R M$ is injective.

We use the notation of Lemma \ref{lemma5.6}. By Lemma \ref{lemma5.6}, since tensoring with $A$ preserves the exactness of 
\[\begin{CD}
0 @>>> R @>{\alpha}>> M @>>> Q @>>> 0, \\
\end{CD}\]
$\id_A \otimes \phi$ is a cocycle in $\Hom_A(G_\bullet,A)$ representing the short exact sequence
\[\begin{CD}
0 @>>> A @>>> {A \otimes M} @>>> {A \otimes Q} @>>> 0.
\end{CD}\]
The map $\alpha$ is cl-\ph\  if and only if $\phi \in \im(\Hom_R(P_0,R) \to \Hom_R(P_1,R))_{\Hom_R(P_1,R)}^\cl$. This holds if and only if $\id_A \otimes \phi \in \im(\Hom_A(G_0,A) \to \Hom_A(G_1,A))$, i.e. if and only if $\id_A \otimes \phi$ is a coboundary in $H^1(\Hom_A(G_\bullet,A))$. By Lemma \ref{lemma5.6}, this holds if and only if there is a map $\gamma:A \otimes M \to A$ such that $\gamma \circ (\id_A \otimes \alpha)=\id_A$.
\end{proof}

\begin{prop}
\label{algebraaxiomforalgebraclosures}
Let $R$ be a domain, and cl be a closure operation coming from a directed family of torsion-free algebras $\scr{A}$. Suppose that a map $\alpha:R \hookrightarrow M$ sending $1 \mapsto u$ is cl-phantom. Then the map $\alpha':R \to \symt(M)$ sending $1 \mapsto u \otimes u$ is also cl-phantom. Hence cl satisfies the Algebra Axiom
\end{prop}

\begin{proof}
Since $\alpha$ is cl-phantom, for some $A \in \scr{A}$, there is a map $\gamma:A \otimes M \to A$ such that $\gamma \circ (\id_A \otimes \alpha)=\id_A$. Define $\lambda:A \otimes \symt(M) \to A$ by $\lambda(a \otimes (m \otimes n))=a\gamma(m)\gamma(n)$. Let $u$ be the image of 1 in $M$. Then 
\[
(\lambda \circ (\id_A \otimes \alpha'))(1)=\lambda(1 \otimes (u \otimes u))=1.
\]
Hence $\alpha'$ is cl-phantom.
\end{proof}

We emphasize the following corollary:

\begin{corollary}
\label{bigcmalgebraaxiom}
Let $B$ be a big \CM\ algebra over a local domain. Then $\cl_B$ is a Dietz closure that satisfies the Algebra Axiom.
\end{corollary}

\begin{proof}
Since $B$ is a big \CM\ module, $\cl_B$ is a Dietz closure by \cite{dietz}. Since $B$ is torsion-free, $\cl_B$ also satisfies the Algebra Axiom by Proposition \ref{algebraaxiomforalgebraclosures}.
\end{proof}

The above relies on the elements of $\scr{A}$ being algebras, rather than modules. We do not know of a simpler condition for a map to be a cl-phantom extension when cl is a module closure, though we have the following:

\begin{lemma}
Let $R$ be a domain, $W$ a torsion-free $R$-module, and $\alpha:R \hookrightarrow M$ an $R$-module map with $M$ \fg. If \[\alpha'=(\id_W \otimes \alpha):W \to W \otimes M\] splits or is pure, then $\alpha$ is $\cl_W$-phantom.
\end{lemma}

\begin{proof}
Let notation be as in Notation \ref{notation1}.
We have the following commutative diagram:

\[\begin{CD}
0 @>>> R @>{\alpha}>> M @>>> Q @>>> 0 \\
@AAA @AA{\phi}A @AA{\psi}A @AA{\id}A @. \\
P_2 @>>> P_1 @>{d}>> P_0 @>>> Q @>>> 0. \\
\end{CD}\]

By definition, $\alpha$ is $\cl_W$-\ph\  if and only if 
\[\phi \in \left(\im(\Hom_R(P_0,R) \to \Hom_R(P_1,R))\right)^{\cl_W}_{\Hom_R(P_1,R)}.\]
This holds if and only if, for every $w \in W$, 
\[w \otimes \phi \in \im(W \otimes \Hom_R(P_0,R) \to W \otimes \Hom_R(P_1,R)).\]
We can identify $W \otimes \Hom_R(P_0,R)$ with $\Hom_R(P_0,W)$ and $W \otimes \Hom_R(P_1,R)$ with $\Hom_R(P_1,W)$. Under this identification, $w \otimes \phi \mapsto \phi_w$, where $\phi_w(y)=\phi(y)w \in W$. So $\alpha$ is $\cl_W$-\ph\  if and only if for every $w \in W$, $\phi_w=\lambda_w \circ d$ for some $\lambda_w:P_0 \to W$.
We have the following commutative diagram for each $w \in W$:
\[\begin{CD}
0 @>>> W @>{\alpha'}>> {W \otimes_R M} @>>> {W \otimes_R Q} @>>> 0 \\
@AAA @AA{\phi_w}A @AA{\psi_w}A @AA{\id_w}A @. \\
P_2 @>>> P_1 @>{d}>> P_0 @>>> Q @>>> 0. \\
\end{CD}\]
Here $\psi_w(y)=w \otimes \psi(y)$ and $\id_w(z)=w \otimes z$. We know that $\alpha'$ is injective because $W$ is torsion-free.

Suppose that $\alpha'$ splits. Then there is some map $\beta:W \otimes M \to W$ such that $\beta \circ \alpha'=\id_W$. For $w \in W$, define $\lambda_w:P_0 \to W$ to be $\beta \circ \psi_w$. 
Then we have:
\[\phi_w=\beta \circ \alpha' \circ \phi_w=\beta \circ \psi_w \circ d=\lambda_w \circ d.\]

In the case that $\alpha'$ is pure, given any $u_1,\ldots,u_k \in W \otimes M$, we have a splitting of $W \to \im(W) +Ru_1+\ldots+Ru_k$. In particular, for each $w \in W$, we have a map $\beta_w:\im(\psi_w) \to W$ such that $\beta_w \circ \alpha'=\id_W$. Then we can define $\lambda_w:P_0 \to W$ to be $\beta_w \circ \psi_w$.
\end{proof}

More generally, there are two ways to think about cl-\ph\  maps, where $\cl=\cl_W$ is a module closure. First, using notation as above, notice that $\alpha$ is cl-\ph\  if and only if 
\[W \otimes \phi \subseteq \im(W \otimes \Hom_R(P_0,R) \to W \otimes \Hom_R(P_1,R)).\]
Then this holds if and only if for each $w \in W$ there are finitely many maps $\lambda_i:W \otimes d(P_1) \to W$ and corresponding elements $w_i \in W$ such that $w \otimes \phi=\sum \lambda_i \circ (w_i \otimes d)$.

Second, we can identify $W \otimes \Hom_R(P_i,R)$ with $\Hom_R(P_i,W)$, since $P_0$ and $P_1$ are free. Under this identification, $w \otimes \phi$ becomes $\phi_w$, the map that sends $z \mapsto \phi(z)w$. Then the statement that $\alpha$ is \ph\  is equivalent to the existence of maps $\lambda_w:P_0 \to W$ such that $\phi_w=\lambda_w \circ d$ for each $w \in W$. 

\begin{remark} 
These maps may not glue together away from the image of $d$, since in general $W \otimes \Hom(P_i,R)$ is not isomorphic to $\Hom(W \otimes P_i,W)$. When they glue together, the map $\id_W \otimes \alpha$ is split.
\end{remark}

\subsubsection{Heitmann's Closure Operations for Mixed Characteristic Rings}

In his paper on the Direct Summand Conjecture \cite{heitmann}, Heitmann defines two closure operations that do not quite fit the pattern of the closure operations above. We assume that $R$ is a domain of mixed characteristic, with residual characteristic $p>0$.

\begin{citeddefn}[\cite{heitmann}]
For $N \subseteq M$, an element $u \in N_M^{\epf}$, the \textit{full extended plus closure of $N$ in $M$} if there is some $c \ne 0 \in R$ such that for all $n \in \Z_+$, 
\[c^{1/n}u \in \im\left((R^+ \otimes N + R^+ \otimes p^nM) \to R^+ \otimes M\right).\]

The \textit{full rank one closure of $N$ in $M$} is defined similarly: $u \in N_M^{rlf}$ if for every rank one valuation of $R^+$, every $n \in \Z_+$, and every $\epsilon>0$, there exists $d \in R^+$ with $v(d)<\epsilon$ such that 
\[du \in \im\left((R^+ \otimes N + R^+ \otimes p^nM) \to R^+ \otimes M\right).\]
\end{citeddefn}

As discussed in Section 7 of \cite{dietzclosureprops}, we do not know whether full extended plus closure and full rank one closure are Dietz closures. However, they still satisfy the Algebra Axiom.

We focus on a definition of \ph\  for the full extended plus closure, as full rank one closure will be similar. In this case, we need a new version of Lemma \ref{lemma5.6}. Let $\alpha:R \to M$ be an injective map, and use the notation of Notation \ref{notation1}. Notice that if we have a map $\gamma:M \to R$ such that $\gamma \circ \alpha=c^{1/n}\id_R$ for every $n \in \Z_+$, then by Lemma \ref{lemma5.6},
\[c^{1/n}\phi \in \im(R^+ \otimes B \to R^+ \otimes \Hom_R(P_1,R)),\]
 where $B$ is the module of coboundaries in $\Hom_R(P_1,R)$. This image is contained in 
\[\im\left((R^+ \otimes B + R^+ \otimes p^n\Hom_R(P_1,R)) \to R^+ \otimes \Hom_R(P_1,R)\right).\] Since this holds for every $n$, $\phi \in B^{\epf}_{\Hom_R(P_1,R)}$. However, the reverse implication is no longer true. Instead we get the following result:

\begin{lemma}
\label{epflemma}
Let $R$, $\alpha$, $\phi$, $P_\bullet$, etc. be as above, and $B$ the submodule of coboundaries in $\Hom_R(P_1,R)$.
For each $c \in R - \{0\}$ and $n \in \Z_+$, 
\[c^{1/n}\phi \in \im\left((B+p^n\Hom_R(P_1,R)) \to \Hom_R(P_1,R)\right)\] if and only if there is a map $\gamma:M \to R/p^nR$ such that $\gamma \circ \alpha=\overline{c^{1/n}\id_R}$ where $\ \bar{}$ denotes image modulo $p^n$.
\end{lemma}

\begin{proof}
Observe that $c^{1/n}\phi$ is in this image if and only if there exist $\lambda:P_0 \to R$, $\delta:P_1 \to R$ such that 
\[c^{1/n}\phi=(\lambda \circ d)+p^n\delta.\]
 This holds if and only if 
 \[c^{1/n}\phi-(\lambda \circ d) \in p^n\Hom_R(P_1,R).\]
  This is true if and only if the map 
  \[\overline{c^{1/n}\id_R} \oplus \overline{\lambda}:R \oplus P_0 \to R/p^nR\]
   kills $\{\phi(u)-d(u) : u \in P_1\}$. Giving this map is equivalent to giving a map 
   \[\gamma: (R \oplus P_0)/\{c^{1/n}\phi(u)-d(u) : u \in P_1\} \to R/p^nR\]
    such that $\gamma \circ \alpha=\overline{c^{1/n}\id_R},$ and \[M \cong (R \oplus P_0)/\{c^{1/n}\phi(u)-d(u):u \in P_1\}.\]
\end{proof}

This lemma allows us to give an alternate definition of the term ``epf-phantom." 

\begin{prop}
A map $R \overset{\alpha}\to M$ is epf-\ph\  if there is some $c \in R-\{0\}$ such that for every $n \in \Z_+$, there is a map $\gamma_n:R^+ \otimes M \to R^+/p^nR^+$ such that $\gamma_n \circ \alpha^+=\overline{c^{1/n}\id_{R^+}}$, where $\alpha^+=\id_{R^+} \otimes \alpha$ and $\ \bar{}$ denotes image modulo $p^n$.
\end{prop}

\begin{proof}
Notice that $\alpha^+$ is injective, so we can apply Lemma \ref{epflemma} with $R^+$, $\id_{R^+}$, etc. By the lemma, $\gamma_n$ exists if and only if 
\[c^{1/n}\phi \in \im\left((R^+ \otimes B + R^+ \otimes p^n\Hom_R(P_1,R)) \to R^+ \otimes \Hom_R(P_1,R)\right).\] So we have a map $\gamma_n$ for each $n$ if and only if $c^{1/n}\phi \in B^{\epf}_{\Hom_R(P_1,R)}$, i.e., if and only if $\phi$ is epf-phantom.
\end{proof}

\begin{remark}
The result for rlf is very similar--in this case, we have maps $\gamma_{\epsilon,n}$, where $n \in \Z_+$ and $\epsilon>0$.
\end{remark}

\begin{prop}
If a map $\alpha:R \to M$ sending $1 \mapsto u$ is epf-phantom, then so is the map $\alpha':R \to \symt(M)$ sending $1 \mapsto u \otimes u$. Consequently, epf satisfies the Algebra Axiom.
\end{prop}

\begin{proof}
Suppose that $\alpha:R \to M$ is phantom. Then there is a $c \in R-\{0\}$ such that for every $n \in \Z_+$, there is a map $\gamma_n:R^+ \otimes M \to R^+/p^nR^+$ with $\gamma_n \circ \alpha^+=\overline{c^{1/n}\id_R}$. We need to find an appropriate $d \in R-\{0\}$ to show that $\alpha':R \to \symt(M)$ is phantom. Let $d=c^2$. Define $\gamma_n':R^+ \otimes \symt(M) \to R^+/p^nR^+$ by
\[\gamma_n'(s \otimes (m \otimes m'))=\overline{s}\gamma_n(m)\gamma_n(m').\]
 Then 
 \[\gamma_n'(\alpha^+(1))=\gamma_n'(1 \otimes (e_1 \otimes e_1))=\gamma_n(e_1)^2=\overline{c^{2/n}}=\overline{(c^2)^{1/n}}=\overline{d^{1/n}},\] as desired.
\end{proof}

\subsubsection{Mixed characteristic pullback tight closure}

For this subsection, assume that $R$ is reduced and of mixed characteristic, with residual characteristic $p>0$.

\begin{defn}
\label{pullbacktightclosure}
Define a closure cl  on $R$ by $u \in N_M^\cl$ if $\bar{u} \in (N/pN)^*_{M/pM}$, where the asterisk denotes tight closure in the characteristic $p$ setting. We call this closure \textit{pullback tight closure}.
\end{defn}

\begin{lemma}
Suppose that $R/pR$ is reduced, and $\bar{\alpha}:R/pR \to M/pM$ is injective. Then $F^e(\overline{\alpha}):F^e(R/pR) \to F^e(M/pM)$ is injective for all $e \ge 0$.
\end{lemma}

\begin{proof}
Replace $R$ by $R/pR$, and $M$ by $M/pM$. By assumption, $R$ is reduced. Let $W$ be the multiplicative system of \nzd s of $R$, so that $Q=W^{-1}R$ is the total quotient ring of $R$. The map $Q \to W^{-1}M$ is injective. Since $Q$ is a product of fields, $W^{-1}M$ is a product of vector spaces over these fields, and so $Q \to W^{-1}M$ splits. Hence $F^e(Q) \to F^e(Q^{-1}M)$ is injective for all $e \ge 0$. The restriction of this map to $F^e(R)$ has image in $F^e(M)$, and will still be injective, as desired.
\end{proof}

We then get the following lemma:

\begin{lemma}
\label{mixedcharlemma}
Let cl denote pullback tight closure as defined above. Suppose that $\alpha:R \to M$ and $\overline{\alpha}:R/pR \to M/pM$ are injective and $R/pR$ is reduced.
Then $\alpha$ is cl-\ph\  if and only if there is some $c \in R/pR-\{0\}$ such that for every $e \ge 0$, there is a map $\gamma_e:F^e(M/pM) \to R/pR$ such that $\gamma_e \circ F^e(\overline{\alpha})=c \cdot \id_{R/pR}$.
\end{lemma}

\begin{proof}
Let $P_\bullet$ be a resolution of $Q=M/\im(R)$. Then we have a commutative diagram
\[\begin{CD}
0 @>>> R @>{\alpha}>> M @>>> Q @>>> 0 \\
@AAA @AA{\phi}A @AAA @AA{\id}A @. \\
P_2 @>>> P_1 @>{d}>> P_0 @>>> Q @>>> 0 \\
\end{CD}\]
Taking the tensor product of this diagram with $F^e(R/pR)$, the top row remains exact by assumption.
By Lemma \ref{lemma5.6}, $\gamma_e$ exists if and only if $cF^e(\overline{\phi})$ is a coboundary. So $\gamma_e$ exists for all sufficiently large $e$ if and only if for each $e \gg 0$, 
\[c\im(F^e(\overline{\phi})) \subseteq \im(\Hom_{R/pR}(P_0/pP_0,R/pR) \to \Hom_{R/pR}(P_1/pP_1,R/pR)).\]
 This holds if and only if 
\[\phi \in (\im(\Hom(P_0,R) \to \Hom(P_1,R)))^\cl_{\Hom(P_1,R)},\]
 i.e., if and only if $\alpha$ is \ph\  by the homological definition.
\end{proof}

\begin{prop}
Let cl denote pullback tight closure. If $\alpha:R \to M$ is cl-phantom, $\overline{\alpha}:R/pR \to M/pM$ is injective, and $R/pR$ is reduced, then $\alpha':R \to \symt(M)$ is cl-phantom.
\end{prop}

\begin{proof}
Using Lemma \ref{mixedcharlemma}, we can define $\gamma_e':F^e(\symt(M)/p\symt(M)) \to R/pR$ by 
\[\gamma_e'(\overline{m}^q \otimes \overline{n}^q)=\gamma_e(\overline{m}^q)\gamma_e(\overline{n}^q),\]
 where $\gamma_e:F^e(M/pM) \to R/pR$ is as in the lemma. Notice that $\symt(M)/p\symt(M) \cong \symt(M/pM)$, which allows us to use the maps $\gamma_e$ to define $\gamma_e'$.
\end{proof}

Note that this closure operation is not generally a Dietz closure.

\begin{eg}
\label{mixedcharpullbacknotdietz}
Let $R=V[x_2,\ldots,x_d]$, where $(V,pV)$ is a \dvr. Then $0$ is not closed in $R$: for any $u \in pR \ne 0$, $\bar{u} \in 0^*_{R/pR}=0$. Since $0^\cl_R=0$ for all Dietz closures cl, mixed characteristic pullback tight closure is not a Dietz closure.
\end{eg}

\subsubsection{Closures constructed from other closures}

The results below describe cases in which if every closure operation in a family satisfies the Algebra Axiom, so does a closure constructed from the family. The constructions are among those that appear in \cite{epstein}. We use the notation of Lemma \ref{lemma5.6}.

\begin{prop}
\label{intersectionalgebraaxiom}
Let $\{\cl_\lambda\}_{\lambda \in \Lambda}$ be a set of closure operations, and define the closure operation $\cl$ by $N_M^\cl = \cap_{\lambda \in \Lambda} N_M^{\cl_\lambda}$. Suppose that every $\cl_\lambda$ satisfies the Algebra Axiom (Axiom \ref{algebraaxiom}). Then cl also satisfies it.
\end{prop}

\begin{remark}
In other words, this result tells us that the Algebra Axiom is intersection stable as defined in \cite[Section 4]{dietzclosureprops}.
\end{remark}

\begin{proof}
Suppose that $\phi \in \im(\Hom(P_0,R) \to \Hom(P_1,R))^\cl$. It suffices to show that this forces $\psi \in \im(\Hom(G_0,R) \to \Hom(G_1,R))^\cl$. By our supposition, we know that 
\[\phi \in \im(\Hom(P_0,R) \to \Hom(P_1,R))^{\cl_\lambda}\] for each $\lambda \in \Lambda$. Since each $\cl_\lambda$ satisfies the axiom, $\psi \in \im(\Hom(G_0,R) \to \Hom(G_1,R))^{\cl_\lambda}$ for all $\lambda$, which immediately gives us the result we want.
\end{proof}

\begin{prop}
Let $\{\cl_\lambda\}_{\lambda \in \Lambda}$ be a directed set of closure operations satisfying the axiom. Then the closure operation $\cl$ defined by $N_M^\cl=\sum_{\lambda \in \Lambda} N_M^{\cl_\lambda}$ also satisfies this axiom.
\end{prop}

\begin{proof}
Note that since $R$ is Noetherian and $M$ is finitely generated over $R$, for each $N \subseteq M$ there is a $\lambda \in \Lambda$ such that $N_M^\cl=N_M^{\cl_\lambda}$ \cite{epstein}. Suppose that 
\[\phi \in \im(\Hom(P_0,R) \to \Hom(P_1,R))^\cl.\] Then for some $\lambda \in \Lambda$, \[\phi \in \im(\Hom(P_0,R) \to \Hom(P_1,R))^{\cl_\lambda}.\] Hence 
\[\psi \in \im(\Hom(G_0,R) \to \Hom(G_1,R))^{\cl_\lambda} \subseteq \im(\Hom(G_0,R) \to \Hom(G_1,R))^\cl.\]
Hence the axiom holds for $\cl$.
\end{proof}

\begin{prop}
\label{preimageclosure}
Let $\phi:R \to S$ be a ring map, and $\cl'$ a closure operation on $S$ satisfying the Algebra Axiom (Axiom \ref{algebraaxiom}). Define a closure operation $\cl$ on $R$ by 
\[N_M^\cl=\left\{x \in M : 1 \otimes x \in \left(\im\left(S \otimes_R N \to S \otimes_R M\right)\right)_{S \otimes_R M}^{\cl'}\right\}.\] If $S$ satisfies the hypotheses of Lemma \ref{staysinjective} (in particular, $R$ a domain and $S$ torsion-free is sufficient), then $\cl$ satisfies the Algebra Axiom as well.
\end{prop}

\begin{remark}
We call closures defined in this way pullback closures. Mixed characteristic pullback tight closure as in Definition \ref{pullbacktightclosure} is one example of a pullback closure, with $S=R/pR$ and $\cl'=*$.
\end{remark}

\begin{proof}
Assume that $\alpha:R \to M$ is cl-phantom. By Lemma \ref{staysinjective}, $\id_S \otimes \alpha:S \to S \otimes_R M$ is injective. We claim that it is $\cl'$-\ph. Since $\alpha$ is cl-\ph, using Notation \ref{notation1}, 
\[\phi \in (\im(\Hom_R(P_0,R) \to \Hom_R(P_1,R)))^\cl_{\Hom_R(P_1,R)}.\]
this implies that 
\[1 \otimes \phi \in (\im(S \otimes_R \Hom_R(P_0,R) \to S \otimes_R \Hom_R(P_1,R)))^{\cl'}_{S \otimes_R \Hom_R(P_1,R)}.\]
We have $S \otimes_R \Hom_R(P_i,R) \cong \Hom_S(S \otimes P_i,S)$ for $i=0,1$. This isomorphism takes $1 \otimes \phi \to \id_S \otimes \phi$. Thus
\[\id_S \otimes \phi \in (\im(\Hom_S(S \otimes P_0,S) \to \Hom_S(S \otimes P_1,S)))^{\cl'}_{\Hom_S(S \otimes P_1,S)}.\]
By Lemma \ref{lemma5.6}, since $\id_S \otimes \alpha:S \to S \otimes_R M$ is injective, this implies that $\id_S \otimes \alpha$ is $\cl'$-\ph.
Since $\cl'$ satisfies the Algebra Axiom (Axiom \ref{algebraaxiom}), this implies that the map $(\id_S \otimes \alpha)':S \to \symt(S \otimes_R M)$ is also phantom. Using the notation from Lemma \ref{lemma5.6} applied to $(\id_S \otimes \alpha)'$, this implies that $\id_S \otimes \psi \in \Hom_S(S \otimes G_0,S)^{\cl'}_{\Hom_S(S \otimes G_1,S)}$. We have $\Hom_S(S \otimes G_i,S) \cong S \otimes \Hom_R(G_i,R)$, 
and the isomorphism takes $\id_S \otimes \psi \to 1 \otimes \psi$. By the definition of cl, this tells us that $\cl$ satisfies the Algebra Axiom.
\end{proof}

One special case is the case where $\cl'$ is the identity closure on $S$, which is Construction 3.1.1 from \cite{epstein}. The resulting closure on $R$ is the algebra closure $\cl_S$: 
\[N_M^\cl=\{x \in M : 1 \otimes x \in \im(S \otimes N \to S \otimes M)\}.\] 
We proved in Proposition \ref{algebraaxiomforalgebraclosures} that this closure operation satisfies the Algebra Axiom when $S$ is torsion-free over a domain $R$.

\subsection{Partial algebra modifications and phantom extensions}
\label{partialalgmodspreservephantomness}

In this section we show that if cl is a Dietz closure on $R$ satisfying the Algebra Axiom, a partial algebra modification of a cl-\phex\ of $R$ is also a cl-\phex\ of $R$. Partial algebra modifications are found in \cite{bigcmalgapps} as part of a construction of big \CM\ algebras in \charp.

\begin{lemma}
\label{partialalgmodsphantom}
Let $R$ be a local domain, cl a Dietz closure on $R$, $M$ a \fg\ $R$-module with $\alpha:R \to M$ a cl-\phex, and $x_1,\ldots,x_{k+1}$ part of a \sop\ for $R$. Suppose that $x_{k+1}m_{k+1}=\sum_{i=1}^k x_im_i$ for some $m_1,\ldots,m_{k+1} \in M$. Let 
\[M'=M[X_1,\ldots,X_k]_{\le 1}/RF,\]
where
\[F=m_{k+1}-\sum_{i=1}^k x_iX_i.\] By $M[X_1,\ldots,X_k]_{\le 1}$ we denote the module generated by all $m \in M$, $X_1,\ldots,X_k$, and $mX_i$ for $m \in M$ and $1 \le i \le k$. We have a map $\alpha':R \to M \to M'$, where the map $M \to M'$ takes $m \mapsto m$.
Then $\alpha':R \to M'$ is a cl-\phex.
\end{lemma}

\begin{proof}
We return to the notation of Notation \ref{notation1}. As Dietz does in \cite{dietz}, we build a resolution of $Q'$, where $Q'=M'/\im(R)$. Let $w_1,w_2,\ldots,w_n$ be a set of generators for $M$, not necessarily minimal, with $w_1=\alpha(1)$ and $w_n=m_{k+1}$. Then a presentation of $Q'$ is
\[
\begin{CD}
{R^{m(k+1)+1}} @>{\nu'}>> {R^{nk+n-1}} @>{\mu'}>> Q' @>>> {0,}
\end{CD}
\]
where $\nu'$ is given by the matrix
\[
\left(
\begin{array}{c c c c | c}
 & & & & 0 \\
\mbox{\huge $\nu$} & \mbox{\huge 0} & \ldots & \ldots & \vdots \\
& & & & 0 \\
& & & & 1 \\
\hline \\
& & & & x_1 \\
& & & & 0 \\
\mbox{\huge $0$} & \mbox{\huge $\nu_1$} & \mbox{\huge 0} & \ldots & \vdots \\
& & & & 0 \\
\vdots & \ddots & \vdots & & \vdots \\
& & & & x_k \\
& & & & 0 \\
\mbox{\huge 0} & \ldots & \mbox{\huge 0} & \mbox{\huge $\nu_1$} & \vdots \\
& & & & 0 \\
\end{array}
\right),
\]
The corresponding matrix $\nu_1'$ in a presentation of $M'$ is this matrix with the top row of $\nu_1$ followed by 0's added to the top. Note that there are $m$ columns for each of $1,X_1,\ldots,X_k$, and one additional column for the relation given by $F$.

To see that $M'$ is a cl-\phex\ of $R$, we first need to show that $\alpha'$ is injective. It is enough to show that $\beta:M \to M'$ is injective. Suppose that $u \in M$ maps to 0. Then $u=r(m_{k+1}-\sum_{i=1}^k x_iX_i)$. This forces $rx_i=0$ for all $i$. Since $R$ is a domain, $r=0$. But then $u=rm_{k+1}=0$.

To finish, it suffices to show that the top row of $\nu_1'$ is in the closure of the image of the other rows. Let $\bm{x},\bm{y},$ and $H$ be as in \cite[Notation 3.5]{dietz}, let $I=(x_1,\ldots,x_k)$, and let $E_{X^\alpha}$ denote the $n \times m(k+1)$ matrix that has an $n \times n$ identity matrix in the columns corresponding to $X^\alpha$ and 0's elsewhere. Then we need to show that $\bm{x}E_1 \oplus 0$ is contained in
\[
 \left((HE_1 \oplus 0)+R(\bm{y}E_1 \oplus 1)+\sum_{i=1}^k \left(R(\bm{x}E_{X_i} \oplus x_i)+(HE_{X_i} \oplus 0)+R(\bm{y}E_{X_i} \oplus 0)\right)\right)^\cl_{R^{m(k+1)+1}}.
\]
By the proof of \cite[Proposition 3.15]{dietz} and \cite[Lemma~3.1.b]{dietzclosureprops}, we have
\[\bm{x}E_1 \oplus 0 \in \left((HE_1 \oplus 0)+I(\bm{y}E_1 \oplus 0)\right)^\cl_{R^{m(k+1)+1}}.\]
So it suffices to show that $(HE_1 \oplus 0)+I(\bm{y}E_1 \oplus 0)$ is contained in the closure of 
\begin{equation}
\label{algmodmess}
(HE_1 \oplus 0)+R(\bm{y}E_1 \oplus 1)+\sum_{i=1}^k \left(R(\bm{x}E_{X_i} \oplus x_i)+(HE_{X_i} \oplus 0)+R(\bm{y}E_{X_i} \oplus 0)\right).
\end{equation}
It is clear that $(HE_1 \oplus 0)$ is in (\ref{algmodmess}). To see that $I(\bm{y}E_1 \oplus 0)$ is in the closure of (\ref{algmodmess}), let $r \in I$, say $r=-r_1x_1-\ldots-r_kx_k$. Then 
\[r(\bm{y}E_1 \oplus 0)=r(\bm{y}E_1 \oplus 1)+r_1(\bm{x}E_{X_1} \oplus x_1)+\ldots+r_k(\bm{x}E_{X_k} \oplus x_k)-r_1(\bm{x}E_{X_1} \oplus 0)-\ldots-r_k(\bm{x}E_{X_k} \oplus 0).\]
The only parts not obviously contained in the closure of (\ref{algmodmess}) are the $r_i(\bm{x}E_{X_i} \oplus 0)$. However, since $\alpha$ is phantom, by \cite[Lemma 3.14]{dietz} and \cite[Lemma~3.1.b]{dietzclosureprops} we have
\[\bm{x}E_{X_i} \oplus 0 \in \left((HE_{X_i} \oplus 0)+R(\bm{y}E_{X_i} \oplus 0)\right)^\cl_{R^{m(k+1)+1}},\]
which is sufficient.
\end{proof}

\begin{defn}\cite[Definition 5.2.8]{dietzthesis}
\label{limphantom}
Let cl be a closure operation on $R$ and $\alpha:R \to P$ be an injective map, where $P$ may not be \fg\ over $R$. We say that $R \to P$ is a \textit{\lcphex}\ if for all \fg\ $M \subseteq P$ such that $\alpha(R) \subseteq M$, $R \to M$ is a cl-\phex.
\end{defn}

\begin{prop}
\label{limitofphantomislimphantom}
Let cl be a closure operation on $R$ that satisfies the \axioma. Let $P$ be an $R$-module, not necessarily finitely generated, that is a direct limit of \fg\ $R$-modules 
\[R \to M_1 \to M_2 \to \ldots \to P\]
such that $R \hookrightarrow M_i$ is cl-\ph\ for all $i \ge 1$.
Then $P$ is a \lcphex\ of $R$.
\end{prop}

\begin{proof}
Let $M$ be a \fg\ $R$-module such that $R \hookrightarrow M \hookrightarrow P$. Since $M$ is \fg, there is some $i$ such that $\im(M) \subseteq \im(M_i)$. By \cite[Lemma 7.3.3.b]{dietzthesis}, since $R \to M_i$ is cl-phantom, so is $R \to M$. 
\end{proof}

\begin{prop}
Let $R$ be a local domain and let cl be a Dietz closure that satisfies the Algebra Axiom (Axiom \ref{algebraaxiom}). Suppose that $M$ is a \fg\ $R$-module, there is a cl-\ph\ map $\alpha:R \to M$ sending $1 \mapsto e$, and $M'$ is a partial algebra modification of $M$. Then the map $\alpha':R \to M'$ is cl-\ph.

In particular, if $M$ is an $R$-algebra and $e=1$ in $M$, then this result holds.
\end{prop}

\begin{proof}
By Lemma \ref{partialalgmodsphantom}, if $M'=M[X_1,\ldots,X_k]_{\le N}/FM[X_1,\ldots,X_k]_{\le N-1}$ with $N \le 1$, then the result is immediate. If not, let $M_1=M[X_1,\ldots,X_k]_{\le 1}/FR$. By Lemma \ref{partialalgmodsphantom}, $R \to M_1$ is cl-\ph. Since cl satisfies the Algebra Axiom, $R \to \symt(M_1)$ is also cl-\ph, and so $R \to \symm(M_1)/(1-e)\symm(M_1)$ is lim cl-\ph.
 So we have
\[R \to M' \to \symm(M_1)/(1-e)\symm(M_1),\]
where the map $R \to \symm(M_1)/(1-e)\symm(M_1)$ is lim cl-\ph. By Proposition \ref{limitofphantomislimphantom}, $R \to M'$ is cl-\ph.
\end{proof} 

In consequence, our construction of big \CM\ algebras using cl-\phex s could have used partial algebra modifications rather than module modifications and symmetric algebras.

\subsection{Smallest big \CM\ algebra closure}

In this section we show that the closure we get from a big \CM\ algebra constructed as in the proof of Theorem \ref{bigcmalgebras} is the same as the closure we get from a big \CM\ algebra constructed using algebra modifications \cite{bigcmalgapps}, and that both are the smallest big \CM\ algebra closure on the ring.

By Proposition \ref{intersectionalgebraaxiom}, the Algebra Axiom is intersection stable as defined in \cite[Section~4]{dietzclosureprops}.

\begin{corollary}
If $R$ has a Dietz closure that satisfies the Algebra Axiom, then it has a smallest such closure.
\end{corollary}

This is immediate from Proposition \ref{intersectionalgebraaxiom}. We do not know whether this closure is a big \CM\ algebra closure, but we do have a smallest big \CM\ algebra closure for $R$.

\begin{prop}
If $R$ has a big \CM\ algebra (equivalently, a Dietz closure that satisfies the Algebra Axiom), then it has a smallest big \CM\ algebra closure. This closure is equal to the closure $\cl_B$ where $B$ is constructed as in Theorem \ref{bigcmalgebras}. It is also equal to the closure $\cl_B$ where $B$ is constructed using algebra modifications as in \cite{bigcmalgapps}.
  \end{prop}

\begin{proof}
For the second statement, let $B$ be a big \CM\ algebra constructed by the method of Theorem \ref{bigcmalgebras}, and $B'$ another big \CM\ algebra for $R$. We show that for any element of $B'$, there is a map $B \to B'$ whose image contains that element. This is enough by \cite[Proposition~3.6]{dietzclosureprops}. Let $R \to B'$ be any map of $R$-modules. We construct a map $B \to B'$ that extends this map. If at any stage, we have a map $M \to B'$, and we take a module modification of $M$, the map extends as in \cite[Proposition~4.14]{dietzclosureprops}. If we have a map $M \to B'$, and the map $R \to M$ takes $1 \mapsto u$, it extends to a map $\symm(M)/(1-u)\symm(M) \to B'$ as $B'$ is an $R$-algebra. Hence starting with the map $R \to B'$, we can construct a map $B \to B'$ with the necessary properties. This implies both the first and second statements of the Proposition.

For the last statement, it suffices to show that if $B$ is a big \CM\ algebra constructed with algebra modifications, $\cl_B$ is also the smallest big \CM\ algebra closure. Let $B'$ be any big \CM\ algebra. We show that for any element of $B'$, there is a map $B \to B'$ whose image contains that element. We start with any $R$-module map $R \to B'$. Suppose that we have a map $S \to B'$, and that we take an algebra modification
\[S'=S[X_1,\ldots,X_k]/FS[X_1,\ldots,X_k],\]
where $s-x_1X_1-\ldots-x_kX_k$, $x_1,\ldots,x_{k+1}$ are part of a \sop\ for $R$, and $sx_{k+1}=s_1x_1+\ldots+s_kx_k$ is a bad relation in $S$. Since $B'$ is a big \CM\ algebra for $R$, $s \in (x_1,\ldots,x_k)B'$, say
\[s=x_1b_1+\ldots+x_kb_k.\]
Then we can extend the map $S \to B'$ to $S'$ by sending $X_i \mapsto b_i$. This gives us a well-defined map $S' \to B'$. Hence we have a map $B \to B'$ whose image includes the image of the original map $R \to B'$.
\end{proof}

\section{Dietz closures satisfying the Algebra Axiom in \charp}
\label{dietzplusalgincharp}

\begin{prop}
\label{dietzalginbigcmalg}
Let $R$ be a local domain and cl be a Dietz closure on $R$ that satisfies the Algebra Axiom (Axiom \ref{algebraaxiom}). Then cl is contained in a big \CM\ algebra closure $\cl_B$.
\end{prop}

\begin{proof}
This proof follows the proof of \cite[Theorem 5.1]{dietzclosureprops}. We can construct $B$ by first constructing a big \CM\ algebra as in Section \ref{algebraaxiomsection}. Then we can use the second type of module modifications  described in \cite[Definition~4.18]{dietzclosureprops}. At every stage, we have preserved $1 \not\in \im(m)$. Repeating these two steps infinitely many times, we get a big \CM\ algebra $B$ such that $\cl \subseteq \cl_B$.
\end{proof}

\begin{defn}
\label{bigcmalgclosure}
Let $R$ be a complete local domain of \charp\ and let $\cal{B}$ be the family of big \CM\ algebras of a ring $R$. By a result of Dietz \cite{seeds}, this is a directed family of algebras, and so we can define a closure operation $\cl_{\cal{B}}$ as in Definition \ref{algebrafamilyclosure}. We call this the \textit{big \CM\ algebra closure.}
\end{defn}

\begin{thm}
\label{intightclosure}
Let $R$ be a complete local domain (or analytically irreducible excellent local domain) of \charp, and cl a Dietz closure on $R$ that satisfies the Algebra Axiom. Then cl is contained in tight closure (*).
\end{thm}

\begin{proof}
In \charp, tight closure is equal to the closure $\cl_{\cal{B}}$ given in Definition \ref{bigcmalgclosure} \cite[Theorem 11.1]{solidclosure}. Since by Proposition \ref{dietzalginbigcmalg}, cl is contained in $\cl_B$ for some big \CM\ algebra $B$, $\cl \subseteq *$.
\end{proof}

Note that in equal characteristic 0, it is not known whether $\cl_{\cal{B}}$ is a closure operation, though it generates one as in Definition \ref{familygeneratedclosure}. It is known that $*\text{EQ}$, big equational tight closure, is contained in $\cl_{\cal{B}}$ but it is not known whether they are equal, so we cannot currently prove Theorem \ref{intightclosure} in this case.

If $R$ has equal characteristic 0 and we restrict $\cal{B}$ to be the set of big Cohen-Macaulay algebras that are ultrarings in the sense of \cite{schoutens}, then by the main result of \cite{dietzrg}, $\cl_{\cal{B}}$ is a closure operation. However, we don't know whether this closure operation agrees with any version of tight closure either.

The following Corollaries are immediate from Theorem \ref{intightclosure}:

\begin{corollary}
Let $R$ be a complete local domain of \charp. Then tight closure is the largest Dietz closure satisfying the Algebra Axiom on $R$.
\end{corollary}

This result is a partial answer to a question asked in \cite{dietzclosureprops} regarding whether there is a largest Dietz closure on a given ring.

\begin{corollary}
\label{trivialcorollary}
Let $R$ be a complete local domain of \charp, and suppose that $R$ is weakly F-regular (i.e., tight closure is trivial on $R$). Then all Dietz closures on $R$ that satisfy the Algebra Axiom are trivial on $R$.
\end{corollary}

Theorem \ref{algaxiomindependent} below along with Example \ref{mixedcharpullbacknotdietz} implies that the Algebra Axiom is independent of the Dietz axioms.

\begin{thm}
\label{algaxiomindependent}
The Dietz Axioms do not imply the Algebra Axiom, i.e., there exist Dietz closures that do not satisfy the Algebra Axiom.
\end{thm}

\begin{proof}
Let $(R,m,k)$ be a complete local domain of \charp\ that is weakly F-regular but not regular and has dimension $d$. By \cite[Corollary~5.12]{dietzclosureprops}, $R$ has a nontrivial Dietz closure, $\cl=\cl_{\syz^d(k)}$. If cl satisfied the Algebra Axiom, then by Corollary \ref{trivialcorollary}, it would be trivial on $R$. Hence cl is a Dietz closure on $R$ that does not satisfy the Algebra Axiom.
\end{proof}

\begin{defn}
\label{testidealdef}
Let cl be a closure operation on a ring $R$. Define the \textit{cl-test ideal} of $R$ by \[\tau_{\cl}(R)=\bigcap_{N \subseteq M \text{ f.g. }} N:N^\cl_M.\]
\end{defn}

This is a definition that extends the notion of a test ideal for tight closure, inspired by \cite{liftableintegralclosure}. In the case below, they could prove to be interesting objects to study.

\begin{lemma}
Let $R$ be a complete local domain of \charp, and cl a Dietz closure on $R$ that satisfies the Algebra Axiom. Then the cl-test ideal of $R$ is nonzero.
\end{lemma}

\begin{proof}
By \cite{fndtc}, $R$ has at least one nonzero test element for tight closure. Since $\cl \subseteq *$, this will also be a test element for cl. Hence the cl-test ideal of $R$ is nontrivial.
\end{proof}

This should lead to further connections between Dietz closures on a ring $R$ and the singularities of $R$.

\section*{Acknowledgments}

My thanks to Mel Hochster for many helpful conversations relating to this work. I also thank Geoffrey D. Dietz, Linquan Ma, and Felipe Perez for their comments on a draft of the manuscript, and Karen Smith for comments on my thesis that were also applied to this paper. This work was partially supported by NSF grants DGE 1256260 and DMS 1401384.

\section*{Bibliography}

\bibliography{rrgbibfile}{}
\bibliographystyle{plain}
\end{document}